\documentclass[12pt,reqno]{amsart}

\advance \topmargin by -\headheight
\advance \topmargin by -\headsep
     
\usepackage{amsmath}
\usepackage{amsfonts}
\usepackage{stmaryrd}
\usepackage{amssymb}
\usepackage{amsthm}

\usepackage{csquotes}

\usepackage{mathrsfs}
\usepackage{dsfont}
\usepackage{bm}

\usepackage{mathtools}
\DeclarePairedDelimiter{\ceil}{\lceil}{\rceil}

\numberwithin{equation}{section}
\renewcommand\vec{\bm}

\newtheorem{theorem}{Theorem}[section]

\newtheorem{lemma}[theorem]{Lemma}
\newtheorem{Proposition}[theorem]{Proposition}
\newtheorem{Conjecture}[theorem]{Conjecture}

\begin{document}
\title[Arithmetic Combinatorics on Vinogradov Systems]{Arithmetic Combinatorics on \\ Vinogradov Systems}
\author[Akshat Mudgal]{Akshat Mudgal}
\date{} % Activate to display a given date or no date (if empty),
         % otherwise the current date is printed 
\address{Department of Mathematics, Purdue University, 150 N. University Street, West Lafayette, IN 47907-2067, USA}
\email{am16393@bristol.ac.uk, amudgal@purdue.edu}
\maketitle

\begin{abstract}
In this paper, we present a variant of the Balog-Szemer\'edi-Gowers theorem for the Vinogradov system. We then use our result to deduce a higher degree analogue of the sum-product phenomenon. 
\end{abstract}

\section{Introduction}
Given a finite subset $A$ of real numbers, we define the sumset $A+A$ as
$$ A+A = \{ a_1 + a_2 \ | \  a_1, a_2 \in A \}, $$
and the additive energy $E(A)$ as
$$ E(A) = |\{  (a_1,a_2,a_3,a_4) \in A^{4} \ | \ a_1 + a_2 = a_3 + a_4   \}|. $$
It is known that if $A$ is an arithmetic progression, then $A+A$ is a small set and a simple application of the Cauchy-Schwarz inequality shows us that $E(A)$ must be large. However the reverse does not hold, that is, given any finite subset $A$ of real numbers with large additive energy $E(A)$, it is not always true that $|A+A|$ is small. In such a situation, the Balog-Szemer\'edi-Gowers theorem gives us the next best alternative by finding a significantly large subset $A'$ of $A$ such that $A' + A'$ is a small set. It is natural to ask whether this circle of ideas may be generalised to additive energies defined via systems of equations of higher degree. In this note, we establish a variant of Balog-Szemer\'edi-Gowers theorem for the Vinogradov system, that is, the system of equations
\begin{equation} \label{vinos}
\sum_{i=1}^{s} x_i^j  = \sum_{i=1}^{s} y_i^j \ \ (1 \leq j \leq k), 
\end{equation}
where $s$ and $k$ are natural numbers. The number of solutions of this system of equations, when the variables $x_i,y_i$ are restricted to a set $A$, has been widely studied by various authors and has close connections to many other problems in additive number theory (see \cite{Woo2014}). Our objective will be a result of a rather different flavour as compared to the traditional goals of earlier work whose focus lay in bounding the number of solutions to the above system of equations. In fact, given a set $A$ with many solutions to the Vinogradov system, we will obtain a structure theorem showing that one can extract a large subset of $A$ with desirable arithmetic properties. \par

We record some definitions to state our main result. We denote by $J_{s,k}(A)$ the number of solutions to the above system $\eqref{vinos}$ with $x_i,y_i \in A$, and we write
\[ \mathscr{A} =  \{ (a,a^2,\dots, a^k) \ | \ a \in A\} \subseteq \mathbb{R}^k. \]
Furthermore, we define $l \mathscr{A}$ as the $l$-fold sumset of $\mathscr{A}$, that is,
$$l \mathscr{A}  = \{ a_1 + \dots + a_{l} \ | \ a_i \in \mathscr{A} \  \text{for} \ 1 \leq i \leq l \} .  $$
For a finite set $A$ of real numbers, we define its diameter $X_{A} = | \sup A - \inf A|$. Lastly, a key ingredient in our results is an appropriate upper bound for $J_{s,k}(A)$, which, with the current technology, is available for well-spaced sets $A$. We call a finite subset $A$ of real numbers to be \emph{well-spaced} if 
\[  | A \cap (j, j+1]| \leq 1 \ \text{for all} \ j \in [ \inf A - 1, \sup A] \cap \mathbb{Z} . \]
With these definitions in hand, we state our main result as follows. 

\begin{theorem} \label{bsgvmvt}
For every $m \geq 1$, natural numbers $s,k$ such that $s \geq k(k+1)$, and sufficiently small positive number $\epsilon$, the following holds. Suppose that $A$ is a finite subset of real numbers such that $A$ is well-spaced, $|A|$ is sufficiently large and $ X_{A} \leq {|A|}^m$. Further, let $\alpha \leq 1$ be any positive real number such that 
\begin{equation} \label{h1}
 \log{\alpha^{-1}} \leq \frac{\epsilon k(k+1)}{43200}\log{|A|}  \ \text{and} \  J_{s,k}(A) = \alpha {|A|}^{2s-k(k+1)/2}.  \end{equation}
Then there exists $A' \subseteq A$ with $|{A'}| \geq {\alpha}^{\epsilon/4} |A|^{1 - \epsilon k^2}$ such that for all $l \in \mathbb{N}$, we have
\[ |l \mathscr{A'}| \ \leq \  {\alpha}^{-(1 + 2\epsilon(l+k^2))}  |A'|^{\frac{k(k+1)}{2} (1 + \epsilon( l +  k^2 ))}, \]
where $\mathscr{A'} =  \{ (a,a^2,\dots, a^k) \ | \ a \in A'\}$.
\end{theorem}

In the above theorem, it is implicit that $\epsilon$ must be small in terms of $s$ and $k$ while $|A|$ must be large enough in terms of $\epsilon$, $k$ and $m$. We remark that the case $k = 1, s=2$ of our result follows as a straightforward corollary from the Balog-Szemer\'edi-Gowers theorem. Thus our conclusion can be seen as a higher degree analogue of the Balog-Szemer\'edi-Gowers theorem.\par
 
 We now relate estimates for $|l \mathscr{A}|$ to multiplicative properties of the set $A$. Given a finite set $A \subseteq \mathbb{R}\setminus \{0\}$, we define the product set $AA$ and the quotient set $A/A$ as  
\[ AA = \{ a_1 a_2 \ | \ a_1, a_2 \in A \} \ \text{and} \ A/A = \{ a_2/a_1 \ | \ a_1, a_2 \in A \}. \]
We recall Vinogradov's $\ll$ and $\gg$ notation. Thus, the expressions $X \ll Y$ and $Y \gg X$ imply that there exists some absolute constant $C> 0$ such that $|X| \leq C|Y|$. With this notation in hand, we now state our second result.

\begin{theorem} \label{vmvtsp}
Let $m$ and $\epsilon$ be positive real numbers with $\epsilon$ sufficiently small, and let $s,k$ be natural numbers such that $2 \leq k \leq s \leq k(k+1)/2$. Suppose there is a finite set $A \subseteq \mathbb{R} \setminus \{0\}$ such that $A$ is well-spaced, $|A|$ is large enough, and $X_A \leq |A|^m$. Then we have
\begin{equation} \label{vmvtsumprod} 
 |A|^{2s - 2\epsilon}  \ll  |s \mathscr{A} + s \mathscr{A}| |A/A|^{s - 1/2 - \epsilon}.  
\end{equation}
\end{theorem}

We note that if one takes $k=s=1$ in the above result, we get a bound of the shape
\[ |A|^{4- 4\epsilon} \ll |A+A|^2 |A/A|,\]
which is a straightforward consequence of \cite[Lemma $2.3$]{So2009}. This bound holds for all finite sets $A$ of real numbers, and shows that either
\[ |A+A| \gg |A|^{4/3 - \epsilon} \ \text{or} \ |A/A| \gg |A|^{4/3 - \epsilon}. \]
This is known as a type of sum-product phenomenon. We note that $\eqref{vmvtsumprod}$ implies that either 
\[ |s \mathscr{A} + s \mathscr{A}| \gg |A|^{s(1 + \delta)} \gg |s \mathscr{A}|^{1 + \delta} \ \text{or} \ |A/A| \gg |A|^{1 + \delta}, \]
where $\delta = 1/(4s-1) - \epsilon > 0$ is permissible when $A$ and $\epsilon$ satisfy the hypothesis of Theorem $\ref{vmvtsp}$.  Thus our result can be seen as a higher degree analogue of the sum-product phenomenon. In fact, we follow ideas from Solymosi's work on sum-product results \cite{So2009} to prove Theorem $\ref{vmvtsp}$. \par

We comment briefly on the sum product phenomenon as well. This is a large collection of results originating from a paper of Erd\H{o}s and Szemer\'{e}di \cite{ES1983} which conjectured that for all $\delta < 1$ and finite sets of natural numbers $A$, one has
\begin{equation}  \label{sumproddef}
|A+A| + |AA| \gg |A|^{1 + \delta}.
\end{equation}
Many authors have worked on this kind of problem, and in particular, Solymosi \cite{So2009} proved that one can take $\delta < 1/3$ in $\eqref{sumproddef}$. The current best known bound was given by Shakan \cite{Sh2018} which states that $\delta < 1/3 + 5/5277$ is permissible in $\eqref{sumproddef}$.
 \par

We now consider the Vinogradov system with $s,k$ satisfying $k \geq 2$ and $s= k(k+1)/2$. Suppose we have a set $A$ that satisfies the hypothesis of Theorem $\ref{vmvtsp}$ as well as the estimate
\[  |s \mathscr{A} + s \mathscr{A}|  \leq |A|^{k(k+1)/2 + \epsilon}. \]
In this case, $\eqref{vmvtsumprod}$ implies the bound
\[ |A|^{s - 3\epsilon} \ll |A/A|^{s-1/2 - \epsilon}. \]
This gives us 
\begin{equation} \label{endg}
 |A/A| \gg |A|^{1 + \delta}, 
\end{equation}
where $\delta = (1/2 - 2\epsilon)/(s-1/2 - \epsilon) = (1 - 4\epsilon)/(2s - 1 - \epsilon) > 0$.
\par

With estimates for $A/A$, we now use the Pl{\"u}nnecke-Ruzsa theorem to derive estimates for $AA$. Thus, applying Lemma $\ref{prineq}$ for the multiplicative group $\mathbb{R}\setminus \{0 \}$, we get
\[ \frac{|A/A|}{|A|} \leq \bigg( \frac{|AA|}{|A|} \bigg)^2. \]
Combining this with $\eqref{endg}$, we see that
\[ |AA| \gg |A|^{1 + \delta / 2}. \] 

Thus, the above conclusion and Theorem $\ref{bsgvmvt}$ combine in a suitable way after some appropriate rescaling. We record this below.

\begin{theorem}  \label{main}
For every $m \geq 1$, all sufficiently small positive numbers $\epsilon$ and all natural numbers $s,k$ such that $s \geq k(k+1)$ and $k \geq 2$, the following holds. Suppose that $A$ is a finite, well-spaced subset of $\mathbb{R} \setminus \{0\}$ such that $|A|$ is sufficiently large and $ X_{A} \leq {|A|}^m$. Let $\alpha \leq 1$ be any real number such that
\[ \log{\alpha^{-1}} \leq \frac{\epsilon k(k+1)}{43200}\log{|A|}  \ \text{and} \  J_{s,k}(A) = \alpha {|A|}^{2s-k(k+1)/2}. \]
Then we have
\[ |A/A| \gg |A|^{1 + \delta - \epsilon k^2 } \ \text{and} \ |AA| \gg |A|^{1 + \delta/2 - 4 \epsilon k^2}, \] 
for $\delta = 1/(k(k+1) -1) > 0$.
\end{theorem}

Theorem $\ref{vmvtsp}$ and Theorem $\ref{main}$ roughly state that the Vinogradov system of equations controls multiplicative properties of a well-spaced set. Theorem $\ref{vmvtsp}$ controls $|2s\mathscr{A}|$ and $|A/A|$ simultaneously, while Theorem $\ref{main}$ assumes $J_{s,k}(A)$ to be extremally large, and then provides stronger bounds for $|A/A|$ and $|AA|$. The latter is a generalisation of the \enquote{few sums, many products} phenomenon. We note that when $k=1$ and $s=2$, a result of Elekes and Ruzsa \cite{ER2003}, along with a straightforward application of the Balog-Szemer\'edi-Gowers theorem, implies that 
\[ |AA| \gg |A|^{2 - \epsilon} \ \text{and} \ |A/A| \gg |A|^{2 - \epsilon}, \]
is permissible in Theorem $\ref{main}$. Thus, Theorem $\ref{vmvtsp}$ and Theorem $\ref{main}$ can be seen as higher dimensional analogues of different types of sum-product phenomenon.
\par

A key component in our proofs of Theorem $\ref{bsgvmvt}$ and Theorem $\ref{vmvtsp}$ will be suitable estimates for $J_{s,k}(A)$. In the following section, we will show that
\begin{equation} \label{lowbd}  J_{s,k}(A) \gg |A|^s + \frac{1}{|s\mathscr{A}|} |A|^{2s}. \end{equation}
From $\eqref{lowbd}$, we can see that when $A$ is an arithmetic progression, we have
$$ J_{s,k}(A) \gg |A|^s +  |A|^{2s - k(k+1)/2}.$$

As for upper bounds, proving almost sharp estimates for $J_{s,k}(A)$ had been a difficult open problem till recently. Bourgain, Demeter and Guth \cite{BDL2016} first proved suitably  sharp estimates for $J_{s,k}(A)$ and later, Wooley \cite{Woo2017} showed similar results using a different approach. In \S2, we will deduce the following upper bound as a straightforward consequence of Bourgain, Demeter and Guth's results from \cite{BDL2016}.

\begin{theorem} \label{BDG1}
For all $\epsilon >0$ and natural numbers $s$ and $k$, the following holds for all large enough, well-spaced subsets $A$ of real numbers
\[  J_{s,k}(A) \  \leq \ X_A^{\epsilon} (|A|^{s} + |A|^{2s - k(k+1)/2}). \]
\end{theorem}

Now as in Theorem $\ref{bsgvmvt}$, suppose we are also given that the diameter $X_A$ of $A$ is bounded above by an expression that is polynomial in $|A|$. Thus, say there exists a fixed real number $m \geq1$ such that $X_A \leq |A|^m$. In this case, the above upper bound gives us 
\begin{equation}\label{eq:uppbd1} J_{s,k}(A) \  \leq \ {|A|}^{\epsilon} ({|A|}^{s} + {|A|}^{2s - k(k+1)/2}), \end{equation}
upon appropriately rescaling $\epsilon$. Furthermore, it can be seen, via an application of the Cauchy-Schwarz inequality, that 
$$ J_{s,k}(A') \ | s \mathscr{A'}| \ \geq \ |A'|^{2s}.   $$
Using $\eqref{eq:uppbd1}$ with this lower bound, we see that for all $\epsilon > 0$, natural numbers $l$, $k$ with $l \geq k(k+1)/2$, and sets $A'$ as in Theorem $\ref{bsgvmvt}$, one has 
$$ | l \mathscr{A'}| \ \geq \ |A'|^{\frac{k(k+1)}{2} - \epsilon}.$$
But for such sets $A'$, it follows from Theorem $\ref{bsgvmvt}$ that
$$|l\mathscr{A'}| \ \leq \ |A'|^{\frac{k(k+1)}{2} (1 + 2\epsilon (l+k^2))}.  $$
Thus our upper bound for $| l \mathscr{A'}|$ is almost sharp. \par

We further remark that the condition $X_A \leq |A|^m$ for a fixed real number $m$ as $A$ is allowed to grow sufficiently large is used to ensure that 
\[ J_{s,k}(A) \leq X_A^{\epsilon} |A|^{2s - k(k+1)/2} \leq |A|^{\epsilon m} |A|^{2s - k(k+1)/2}, \]
for $s \geq k(k+1)/2$. The second upper bound is what we require for our proof to work, and thus improvements in upper bounds for $J_{s,k}(A)$ will simultaneously improve our results by removing the constraint on sparsity of $A$. In particular, one might conjecture that the following stronger upper bound for $J_{s,k}(A)$ should hold. 

\begin{Conjecture} \label{demcon}
Let $s, k \in \mathbb{N}$ such that $s \geq k(k+1)/2$, and let $\epsilon > 0$ be a positive real number. Then for all finite, non-empty sets $A$ of $\mathbb{Z}$ such that $|A|$ is large enough, we have
\[ J_{s,k}(A) \leq |A|^{2s - k(k+1)/2 + \epsilon}. \]
\end{Conjecture}

When $k = 2, s = 3$, a version of this conjecture appears in \cite[Question 8.6]{De2014}. In the same paper \cite[Proposition 8.7]{De2014}, Demeter shows that for all $\epsilon > 0$, we have
\[ J_{3,2}(A) \leq |A|^{7/2 + \epsilon},\] 
for all large enough sets $A$ of $\mathbb{R}$. In \S6, we establish variants of Theorem $\ref{bsgvmvt}$ and Theorem $\ref{main}$ that do not require the diameter restriction $X_A \leq |A|^m$ for our set $A$, but require a hypothesis different from $J_{s,k}(A) = \alpha |A|^{2s-k(k+1)/2}$.   \par

As previously mentioned, Theorem $\ref{bsgvmvt}$ can be seen as a higher degree analogue of the Balog-Szemer\'edi-Gowers theorem. In fact, the latter essentially works with sets that have many solutions to the linear case $k=1, s=2$ of the Vinogradov system. It was originally proved by Balog and Szemer\'edi \cite{BS1994}, and further developed by Gowers \cite{Gow1998} in his celebrated work on four term arithmetic progressions in subsets of integers. Since Gowers' work, many authors have further refined this result, with the most recent improvement being due to Schoen \cite{Sch2015}. Moreover, this result has also been generalised for the case of many-fold sumsets. This includes results by Sudakov, Szemer\'edi and Vu \cite{SSV2005}, and Borenstein and Croot \cite{BC2011}. We will utilise ideas from Borenstein and Croot's paper along with appropriate estimates for $J_{s,k}(A)$ to establish Theorem \ref{bsgvmvt}. \par

We now outline the structure of our paper. We use \S2 to state some important definitions and preliminary results that we will use throughout our proof. We begin the proof of Theorem $\ref{bsgvmvt}$ in \S3, where we employ the hypothesis of Theorem $\ref{bsgvmvt}$ to construct a dense hypergraph on $\mathscr{A}^s$ which has a small restricted sumset. We then state Lemma $\ref{hypvmvt}$, a result about hypergraphs with desirable additive properties, from which we deduce our main result. Subsequently, \S4 is dedicated to proving Lemma $\ref{hypvmvt}$. In \S5, we prove Theorem $\ref{vmvtsp}$. Lastly, we use \S6 to establish Theorem $\ref{absvmvt}$ and Theorem $\ref{absmain}$, which are variants of our earlier results but with different hypotheses.  \par

{\textbf{Acknowledgements.}} The author's work is supported by a studentship sponsored by a European Research Council Advanced Grant under the European Union's Horizon 2020 research and innovation programme via grant agreement No.~695223. The author is grateful for support and hospitality from the University of Bristol. The author would like to thank Trevor Wooley for his guidance and direction, and the anonymous referee for many helpful comments.  

\section{Preliminaries} \label{prelim}

As before, given a finite subset $A$ of an additive group $G$, we define the sumset $A+A$ as
$$ A+A = \{ a_1 + a_2 \ | \  a_1, a_2 \in A \}, $$
and the difference set $A-A$ as
$$ A-A = \{ a_1 - a_2 \ | \  a_1, a_2 \in A \}. $$
Furthermore, for a given natural number $s$, we use $A^s = A \times A \times \dots \times A$ to denote the Cartesian-product of $s$ copies of $A$. Given $G \subseteq A^s$, we define the restricted sumset $\Sigma(G)$ of $G$ as 
\begin{equation} \label{sigdef}
 \Sigma(G) = \{ a_1 + a_2 + \dots + a_s  \ | \ (a_1,a_2,\dots,a_s) \in G \}.
 \end{equation}
As a special case of the restricted sumset, the iterated sumset $sA$ is defined to be  $\Sigma(A^s)$. Given $A,B \subseteq G$, we let $E(A,B)$ denote the number of solutions to 
$$x_1 + y_1 = x_2 + y_2,$$
with $x_1, x_2 \in A$ and $y_1,y_2 \in B$. \par

In \S4, we will use the following consequence of Balog's refinement \cite[Theorem 5]{Bal2007} of the Balog-Szemer\'edi-Gowers theorem.

\begin{theorem} \label{bsg}
Let $A$ and $B$ be finite subsets of an additive abelian group $G$ with $|A|= |B|$ and $E(A,B) \geq \alpha |A|^3$.  Then there exists $A' \subseteq A$ such that 
\[     |A'| \ \geq  C_1(\alpha) |A| \ \ \text{and} \  \ |A' + A'| \ \leq \  C_2(\alpha)  |A|,   \]
where \[    C_1(\alpha) \ =  \   \frac{3}{2^{19}} \frac{{\alpha}^{3}}{\log{(32/ {\alpha})}} \ \ \text{and} \ \  C_2(\alpha) \ =   \   \frac{2^{45}}{3}  \frac{\log{(32/ {\alpha})}}{{\alpha}^7}. \]
\end{theorem}

We will now assume that $A$ is a finite subset of real numbers. In this note, we will look at Vinogradov's mean value system, that is, the system of equations
\begin{equation}\label{eq:def1}
\sum_{i=1}^{s} x_i^j  = \sum_{i=1}^{s} y_i^j \ \ (1 \leq j \leq k), 
\end{equation}
where $s$ and $k$ are some natural numbers. We define $J_{s,k}(A)$ as the number of solutions to the above system of equations with variables $x_i, y_i$ restricted to the set $A$ for all $1 \leq i \leq s$. Using $\mathscr{A}$ to denote the set 
$$\mathscr{A} = \{ (a,a^2,\dots, a^k) \ | \ a \in A\} \subseteq \mathbb{R}^k, $$ 
it follows that 
\[ J_{s,k}(A) =  \big| \big\{  (a_1,a_2,\dots,a_{2s}) \in {\mathscr{A}}^{2s} \ \big| \   \sum_{i=1}^{s} a_i  = \sum_{i=s+1}^{2s} a_i \   \big\}\big|. \] 
Note that $|\mathscr{A}| = |A|$,  where for a set $X$, we write $|X|$ for the cardinality of $X$. For all $\vec{n} = (n_1,\dots,n_k)$ in $\mathbb{R}^k$, we define $r(\mathscr{A}^s; \vec{n})$ as 
$$ r(\mathscr{A}^s;\vec{n}) = \bigg| \bigg\{ (a_1,a_2,\dots,a_s) \in \mathscr{A}^s \ \big| \ \vec{n} = \sum_{i=1}^{s} a_i \bigg\}\bigg| , $$
and 
$$ r(\mathscr{A}^s) = \sup_{ \vec{n}\in s\mathscr{A} }  r(\mathscr{A}^s;\vec{n}) .$$
\par

We now show that the lower bound $\eqref{lowbd}$ holds for $J_{s,k}(A)$. 
\begin{lemma}
Let $A$ be a finite set of real numbers. Then we have
\[  J_{s,k}(A) \gg |A|^s + \frac{1}{|s\mathscr{A}|} |A|^{2s}. \]
\end{lemma}

\begin{proof}
We first count the number of diagonal solutions to \eqref{eq:def1}, namely those solutions in which $\{ x_1,\dots,x_s \}$ is a permutation of $\{y_1, \dots ,y_s\}$. This gives us the lower bound
\begin{equation}\label{eq:lowbd1}  J_{s,k}(A) \geq |A|^s. \end{equation}
Next, we see that 
\[  \sum_{\vec{n} \in s\mathscr{A}}  r(\mathscr{A}^s;\vec{n}) = |\mathscr{A}^s| =  |\mathscr{A}|^s = |A|^s.\] 
Applying the Cauchy-Schwarz inequality on this, we get
\[  \sum_{\vec{n} \in s\mathscr{A}}  r(\mathscr{A}^s;\vec{n})^2  \sum_{\vec{n} \in s\mathscr{A}} 1 \geq   (\sum_{\vec{n} \in s\mathscr{A}}  r(\mathscr{A}^s;\vec{n}) )^2 = |A|^{2s}.  \] 
We note that
\[  \sum_{\vec{n} \in s\mathscr{A}}  r(\mathscr{A}^s;\vec{n})^2  = J_{s,k}(A), \] 
and thus, we have
\begin{equation} \label{eq:lowbd2}
J_{s,k}(A) \geq \frac{1}{|s \mathscr{A}|}|A|^{2s}.
\end{equation}
Combining $\eqref{eq:lowbd1}$ and $\eqref{eq:lowbd2}$, we prove our lemma. 
\end{proof}

From this point, we will assume that our set $A$ is well-spaced. With this assumption, we can now get suitable upper bounds of $J_{s,k}(A)$. We wish to prove Theorem $\ref{BDG1}$, which we deduce from the following consequence of work of Bourgain, Demeter and Guth \cite[Theorem 4.1]{BDL2016}.

\begin{lemma} \label{BDG2}
Let $s,k \in \mathbb{N}$ such that $2 \leq k$ and $2 \leq s \leq k(k+1)/2$. For each $1 \leq i \leq P $, let $t_i$ be a point in $(\frac{i-1}{P}, \frac{i}{P}]$. Then for each $\epsilon > 0$, each $P^k \leq R$, each ball $B_R = B(c,R) \in \mathbb{R}^k$ and each set of complex coefficients $ \{ \mathfrak{a_i} \}_{1 \leq i \leq P}$, we have
\[  \frac{1}{|B_R|} \int_{\mathbb{R}^k} | \sum_{i=1}^{P} \mathfrak{a_i} e(x_1 t_i + x_2 t_i^2 + \dots + x_k t_i^k) |^{2s} w_{B_R}(x) dx_1 \dots dx_k  \ll  P^{\epsilon} \big(\sum_{i =1}^{P} |\mathfrak{a_i}|^2 \big)^{s}, \]
where 
\[ w_{B_R}(x) = \bigg(  1 + \frac{|x-c|}{R} \bigg)^{-100k}. \]
\end{lemma}

We proceed by following the proof of Theorem $6.4$ in \cite{LP2017}, which is an exposition of the proof of Corollary $4.2$ in \cite{BDL2016}. 
\begin{proof}[Proof of Theorem $\ref{BDG1}$]
We first define some notation for ease of exposition. We use $P$ to denote $\ceil{X_A}$ and $S_A$ to denote the shifted set $A - \inf{A}$. We use $d\vec{x}$ to denote $dx_1\dots dx_k$ when we integrate over $\mathbb{R}^k$ and $\vec{x}$ to denote $(x_1, \dots, x_k) \in \mathbb{R}^k$. For all $\vec{x} \in \mathbb{R}^k$, write
\[ \mathfrak{f}(A; \vec{x}) =  \sum_{a \in S_A}  e\bigg(x_1 \frac{a}{P} + x_2 \frac{a^2}{P^2} + \dots + x_k \frac{a^k}{P^k}\bigg) , \] 
and
\[ \mathfrak{g}(A; \vec{x}) = \sum_{a \in S_A}  e(x_1 a + x_2 a^2 + \dots + x_k a^k). \]

We now show Theorem $\ref{BDG1}$ for $s \leq k(k+1)/2$. We begin by fixing a Schwartz class function $\phi : \mathbb{R}^k \to [0, \infty)$ that satisfies $\phi(\vec{x}) > 0$ for all $\vec{x} \in \mathbb{R}^k$ and whose Fourier transform $\hat{\phi}$ satisfies $\hat{\phi}(\vec{\xi}) > 0$ for all $\vec{\xi} \in \mathbb{R}^k$ and $\hat{\phi}(\vec{\xi}) \geq 1$ for $|\vec{\xi}| \leq 1$. Given such a function $\phi$, we define for any $U>0$, the functions
 \[ \phi_U(\vec{x}) = \phi(x_1 U^{-k}, \dots, x_k U^{-k}) \]
 and
 \[ \psi_U(\vec{x}) = \phi(x_1 U^{-(k-1)}, x_2 U^{ -(k-2)}, \dots, x_k),\]
 for all $\vec{x} \in \mathbb{R}^k$. We observe that for all $U > 0$ and $\vec{x} \in \mathbb{R}^k$, we have
 \[ \psi_U(x_1, \dots, x_k)  = \phi_U(x_1U, \dots, x_k U^k). \]
 
 Due to the rapid decay of our Schwartz function $\phi$, given $P \geq 1$, there exists a ball $B_R$ of radius $P^k \ll R \ll P^k$, such that 
\[ \phi_{P}(\vec{x}) \ll w_{B_R}(\vec{x}) \ \text{for all} \ \vec{x} \in \mathbb{R}^k. \]

As $A$ is well-spaced, for all $i \in \mathbb{Z}$, we have $|S_A \cap (i, i+1]| \leq 1$. Thus, for $0 \leq i \leq P$, if $S_A\cap (i, i+1] = \{a\}$, we set $t_i = a/P$ and  $ \mathfrak{a_i} = 1$, otherwise we set $t_i = (i+1)/P$ and $ \mathfrak{a_i} = 0$. We now use Lemma $\ref{BDG2}$ with $t_i$ and $\mathfrak{a_i}$ as defined, to get
\begin{equation} \label{decoup}
  \frac{1}{P^{k^2}} \int_{\mathbb{R}^k} | \mathfrak{f}(A;\vec{x}) |^{2s} \phi_{P}(\vec{x}) d \vec{x}  \ll  P^{\epsilon} |A|^{s}.
\end{equation}
We perform a change of variables on the left hand side to get
\[ \frac{1}{P^{k(k-1)/2}} \int_{\mathbb{R}^k} | \mathfrak{g}(A; \vec{x})|^{2s} \phi_{P}(Px_1, \dots, P^k x_k) d \vec{x}, \]
which is the same as
\begin{equation} \label{wutname}
 \frac{1}{P^{k(k-1)/2}} \int_{\mathbb{R}^k} |\mathfrak{g}(A; \vec{x}) |^{2s} \psi_{P}(\vec{x})  d \vec{x}. 
\end{equation}
We now expand the $2s$-moment in the above expression to get
\begin{equation} \label{decoup2}
 \frac{1}{P^{k(k-1)/2}}  \sum_{\vec{a} \in S_A^{2s}} \int_{\mathbb{R}^k} e(x_1 \theta_1(\vec{a}) + \dots + x_k \theta_k(\vec{a})) \psi_{P}(\vec{x}) d \vec{x}, 
\end{equation}
where 
\[ \vec{a} = (a_1, \dots, a_{2s}) \in S_A^{2s},\]
and
\[ \theta_j(\vec{a}) = \sum_{i=1}^{s} (a_i^{j} - a_{i+s}^{j}), \]
for all $1 \leq j \leq k$. \par

We note that for each fixed $\vec{a}$, the integral above is the Fourier transform of the function $\phi(\frac{x_1}{P^{k-1}}, \dots, x_k)$, evaluated at the point $(\theta_{1}(\vec{a}), \dots, \theta_{k}(\vec{a}))$. We further remark that
\[\hat{\psi_{P}}(\eta_1, \dots, \eta_k) = P^{k(k-1)/2}\hat{\phi}(P^{k-1} \eta_1, \dots, \eta_k),   \] 
and thus our integral expression in $\eqref{decoup2}$ is equivalent to
\begin{equation} \label{decoup3}
 \sum_{\vec{a} \in S_A^{2s}} \hat{\phi}(P^{k-1} \theta_{1}(\vec{a}), \dots, \theta_{k}(\vec{a})). 
\end{equation}

As per our choice of $\phi$, we note that $\hat{\phi}(\vec{\xi}) > 0$ for all $\vec{\xi} \in \mathbb{R}^k$ and $\hat{\phi}(\vec{\xi}) \geq 1$ for $|\vec{\xi}| \leq 1$. In particular, if $\vec{\xi}= (0 \dots, 0)$, then $\hat{\phi}(\vec{\xi}) \geq 1$. Lastly, we write
\[ T_A =  \{ \vec{a} \in S_A^{2s} \ | \ \theta_j(\vec{a}) = 0 \ \text{for all} \ 1 \leq j \leq k \}. \] 
We note that, by positivity of $\hat{\phi}$, we have
\begin{align*} 
\sum_{\vec{a} \in S_A^{2s}} \hat{\phi}(P^{k-1} \theta_{1}(\vec{a}), \dots, \theta_{k}(\vec{a})) & \geq \sum_{\vec{a} \in T_A} \hat{\phi}(P^{k-1} \theta_{1}(\vec{a}), \dots, \theta_{k}(\vec{a})) \\
& = \sum_{\vec{a} \in T_A} \hat{\phi}(0, \dots, 0) \\
& \geq  \sum_{\vec{a} \in T_A} 1 = J_{s,k}(S_A).
\end{align*}
We observe that the sum in $\eqref{decoup3}$ is an equivalent expression of the left hand side of $\eqref{decoup}$ and thus, we have
\begin{equation} \label{res1}
J_{s,k}(S_A) \leq  \sum_{\vec{a} \in S_A^{2s}} \hat{\phi}(P^{k-1} \theta_{1}(\vec{a}), \dots, \theta_{k}(\vec{a})) \ll P^{\epsilon} |A|^{s}.
\end{equation}
We conclude this portion of the proof by recalling that the Vinogradov system is translation invariant and thus, $J_{s,k}(S_A) = J_{s,k}(A)$. Hence, we have proved Theorem $\ref{BDG1}$ for $s = k(k+1)/2$. \par

We now assume $ s > k(k+1)/2$. We begin by noting that
\[ |\mathfrak{f}(A;\vec{x})| \leq |S_A| = |A|. \] 
Thus, we have
\begin{equation} \label{linft}
 |\mathfrak{f}(A;\vec{x})|^{2s} \leq |A|^{2s - k(k+1)} |\mathfrak{f}(A;\vec{x})|^{k(k+1)} .
\end{equation}
As in the proof for the $s \leq k(k+1)/2$ case, we deduce that
\[ J_{s,k}(A) \leq   \frac{1}{P^{k(k-1)/2}} \int_{\mathbb{R}^k} | \mathfrak{f}(A;\vec{x}) |^{2s} \psi_{P}(\vec{x}) d \vec{x}.\]
We combine this with $\eqref{decoup}$ and $\eqref{linft}$ to get
\begin{equation} \label{res2}
J_{s,k}(A)  \ll P^{\epsilon} |A|^{2s - k(k+1)/2}.  
\end{equation}

We conclude by mentioning that as $A$ is well-spaced, $P \geq |A|/10$, and as $|A|$ becomes large, $P$ becomes large. Thus, we can replace the implicit constant in Vinogradov's notation in $\eqref{res1}$ and $\eqref{res2}$ by $P^{\epsilon}$. Consequently, we get Theorem $\ref{BDG1}$ by rescaling $\epsilon$ appropriately.   
\end{proof}

Lastly, we show the following upper bounds for $r(\mathscr{A}^s)$.

\begin{lemma}
\label{folklore}
For every $m \geq 1$, $\delta > 0$, $(s,k) \in \mathbb{N}^2$ such that $s \geq k(k+1)$ and $k \geq 2$, the following holds. Suppose that $A$ is a finite, well-spaced subset of $\mathbb{R} \setminus \{0\}$ such that $|A|$ is sufficiently large and $ X_{A} \leq {|A|}^m$. Then, we have
\begin{equation} \label{rep} r(\mathscr{A}^s) \leq |A|^{s- \frac{k(k+1)}{2} + \delta}  \ .   \end{equation}
\end{lemma}
\begin{proof}
We use the same notation as defined at the beginning of the proof of Lemma $\ref{BDG2}$. We note that 
$$  r(\mathscr{A}^s) \  \leq \ |A|^{s- k(k+1)} \  r(\mathscr{A}^{k(k+1)}) ,$$
and thus, it suffices to prove the lemma for $ s = k(k+1)$. Hence, we fix $s = k(k+1)$. \par

We fix a Schwartz class function $\phi : \mathbb{R}^k \to [0, \infty)$ that satisfies $\phi(\vec{x}) > 0$ for all $\vec{x} \in \mathbb{R}^k$ and whose Fourier transform $\hat{\phi}$ satisfies $\hat{\phi}(\vec{\xi}) > 0$ for all $\vec{\xi} \in \mathbb{R}^k$ and $\hat{\phi}(\vec{\xi}) \geq 1$ for $|\vec{\xi}| \leq 1$.. \par

Let $\vec{n} = (n_1, \dots, n_k) \in \mathbb{R}^k$. Consider the following integral expression
\begin{equation} \label{open1}
 I =   \frac{1}{P^{k(k-1)/2}} \int_{\mathbb{R}^k} \mathfrak{g}(A; \vec{x})^{s} e(-x_1 n_1 - \dots -x_k n_k) \psi_{P}(\vec{x})d \vec{x}. 
\end{equation}
 We expand the $s$-exponent in $\eqref{open1}$ to get
\begin{equation} \label{open2}
 \frac{1}{P^{k(k-1)/2}}  \sum_{\vec{a} \in S_A^{s}} \int_{\mathbb{R}^k} e(x_1 \theta_1(\vec{a},\vec{n} ) + \dots + x_k \theta_k(\vec{a},\vec{n} )) \psi_{P}(\vec{x}) d \vec{x}, 
\end{equation}
where 
\[ \vec{a} = (a_1, \dots, a_{s}) \in S_A^{s},\]
and
\[ \theta_j(\vec{a},\vec{n} ) = \sum_{i=1}^{s} a_i^{j} - n_j,\] 
for all $1 \leq j \leq k$. As in the previous proof, we realise that the integral expression in $\eqref{open2}$ is the Fourier transform of the function $\phi(\frac{x_1}{P^{k-1}}, \dots, x_k)$, evaluated at the point $(\theta_{1}(\vec{a}, \vec{n}), \dots, \theta_{k}(\vec{a}, \vec{n}))$. We recall that
\[ \hat{\psi_{P}} (\eta_1, \dots, \eta_k) = P^{k(k-1)/2}\hat{\phi}(P^{k-1} \eta_1, \dots, \eta_k),   \] 
and thus our integral expression in $\eqref{open2}$ is equivalent to
\begin{equation} \label{open3}
 \sum_{\vec{a} \in S_A^{s}} \hat{\phi}(P^{k-1} \theta_{1}(\vec{a},\vec{n}), \dots, \theta_{k}(\vec{a},\vec{n})). 
\end{equation}
\par

We now define 
\[ T_A(\vec{n}) =  \{ \vec{a} \in S_A^{s} \ | \ \theta_j(\vec{a}, \vec{n}) = 0 \ \text{for all} \ 1 \leq j \leq k \}. \]
As in the previous proof, due to the suitable properties of the function $\phi$, we see that
\begin{align*}
\sum_{\vec{a} \in S_A^{s}} \hat{\phi}(P^{k-1} \theta_{1}(\vec{a},\vec{n}), \dots, \theta_{k}(\vec{a},\vec{n})) & \geq \sum_{\vec{a} \in T_A(\vec{n})} \hat{\phi}(P^{k-1} \theta_{1}(\vec{a}), \dots, \theta_{k}(\vec{a})) \\
& = \sum_{\vec{a} \in T_A(\vec{n})} \hat{\phi}(0, \dots, 0)  \geq  \sum_{\vec{a} \in T_A (\vec{n})} 1 .
\end{align*}
Let $\mathscr{S}_A = \{ (a,a^2, \dots, a^k) \ | \ a \in S_A \}$. 
Thus for $\vec{n} \in \mathbb{R}^k$, we have
\[ r(\mathscr{S}_A^{s}; \vec{n}) = \sum_{\vec{a} \in T_A (\vec{n})} 1 \leq \sum_{\vec{a} \in S_A^{s}} \hat{\phi}(P^{k-1} \theta_{1}(\vec{a},\vec{n}), \dots, \theta_{k}(\vec{a},\vec{n})) = I. \]
We go back to $\eqref{open1}$ and note that $\psi_{P}(\vec{x}) > 0$ for all $\vec{x} \in \mathbb{R}^k$. Thus, we apply the triangle inequality on the integral $I$ to get
\[I \leq  \frac{1}{P^{k(k-1)/2}} \int_{\mathbb{R}^k} |\mathfrak{g}(A; \vec{x})|^{2(s/2)} \psi_{P}(\vec{x}) d \vec{x}.  \] 
We note that the expression on the right is the same as the integral in $\eqref{wutname}$, with an $s/2$ instead of $s$. The integral in $\eqref{wutname}$ was equivalent to the left hand side of $\eqref{decoup}$. As $ s/2 = k(k+1)/2$, we use this equivalence to show
\[ I \ll P^{\epsilon} |A|^{s/2}, \]
for all $\epsilon > 0$. We set $\epsilon = \delta/2m$ in the above expression and use the fact that $P \ll X_A \leq |A|^m$ to get
\[r(\mathscr{S}_A^{s}; \vec{n}) \ll  |A|^{s/2 + \delta/2},\] 
for all $\vec{n} \in \mathbb{R}^k$ when $s=k(k+1)$. \par

As the Vinogradov system is translation-dilation invariant, the above implies that for all $\vec{n} \in \mathbb{R}^k$, we have
\[r(\mathscr{A}^{s}; \vec{n}) \ll |A|^{s/2 + \delta/2},\] 
whenever $s = k(k+1)$. This implies that 
\[ r(\mathscr{A}^{k(k+1)}) \ll |A|^{k(k+1)/2 + \delta/2},\]
and after majorising the implicit constant of the inequality by $|A|^{\delta/2}$, we show that
\[ r(\mathscr{A}^{k(k+1)}) \leq  |A|^{k(k+1)/2 + \delta}. \]
This concludes our proof.  
\end{proof}

\section{Proof of Theorem \ref{bsgvmvt}} \label{proofth}

In this section and the next, we will use $N$ to denote $|A|$. We begin by using double-counting to show that
\begin{equation}\label{eq:doublecount} \sum_{ \vec{n} \in s\mathscr{A}}  r(\mathscr{A}^s;\vec{n})^2 = J_{s,k}(A).  \end{equation}
Let $S$ denote the set of $\vec{n}$ that are highly-represented in $s\mathscr{A}$, that is, 
\begin{equation} \label{eq:defS} S = \{  \vec{n} \in s\mathscr{A} \ | \   r(\mathscr{A}^s;\vec{n}) \geq {\textstyle{\frac{1}{2}}}\alpha N^{s-k(k+1)/2}  \}. \end{equation}
Note that 
\begin{align*} \sum_{\vec{n} \in s\mathscr{A} \setminus S}  r(\mathscr{A}^s;\vec{n})^2 \ & \leq \  \frac{\alpha}{2}N^{s-k(k+1)/2} \sum_{\vec{n} \in s\mathscr{A} \setminus S}  r(\mathscr{A}^s;\vec{n}) \\ & \leq \  \frac{\alpha}{2}N^{s-k(k+1)/2} \sum_{ \vec{n} \in s\mathscr{A}}  r(\mathscr{A}^s;\vec{n})  \\ & = \  \frac{\alpha}{2}N^{s-k(k+1)/2} N^{s}. \end{align*}
Using the above inequality with \eqref{eq:doublecount} along with the hypothesis \eqref{h1}, we get
\begin{align}\sum_{\vec{n} \in S}  r(\mathscr{A}^s;\vec{n})^2 & = \ J_{s,k}(A) \ -  \sum_{\vec{n} \in s\mathscr{A} \setminus S}  r(\mathscr{A}^s;\vec{n})^2 \nonumber \\ & \geq \ \alpha N^{2s-k(k+1)/2} - \frac{\alpha}{2} N^{2s-k(k+1)/2} \nonumber \\ &   =   \ \frac{\alpha}{2} N^{2s-k(k+1)/2}.  \label{eq:idk1} \end{align}

Upon combining Lemma $\ref{folklore}$ with $\eqref{eq:idk1}$, we get that for all $\delta >0$,
\begin{align*}
 N^{s- \frac{k(k+1)}{2} + \delta} \sum_{\vec{n} \in S}  r(\mathscr{A}^s;\vec{n}) & \ \geq \  r(\mathscr{A}^s) \sum_{\vec{n} \in S}  r(\mathscr{A}^s;\vec{n}) \\ 
&\ \geq \   \sum_{\vec{n} \in S}  r(\mathscr{A}^s;\vec{n})^2  \\
& \ \geq \ \frac{\alpha}{2} N^{s-k(k+1)/2} N^{s}, 
\end{align*}
and thus,
\begin{equation} \label{sizeofG}
\sum_{\vec{n} \in S}  r(\mathscr{A}^s;\vec{n}) \ \geq \ \frac{\alpha}{2} N^{s-\delta}.
\end{equation}
Furthermore, upon combining the definition $\eqref{eq:defS}$ of $S$ along with $\eqref{h1}$ and $\eqref{eq:doublecount}$, we deduce that
$$ {  \bigg( \ \frac{\alpha}{2} N^{s-k(k+1)/2} \ \bigg)}^2  \sum_{\vec{n} \in S}1 \ \leq  \sum_{\vec{n} \in S}  r(\mathscr{A}^s;\vec{n})^2 \ \leq \ J_{s,k}(A) \ = \ \alpha N^{2s-k(k+1)/2}. $$
This gives us
\begin{equation} \label{sizeofSumG}
|S| \ \leq \ \frac{4}{\alpha} N^{k(k+1)/2}. 
\end{equation}
\par

Define $G \subseteq \mathscr{A}^s$ as
$$ G = \{  (a_1,a_2,\dots,a_s) \in \mathscr{A}^s \ | \  \sum_{i=1}^{s} a_i \in S \}.$$
By definition of $G$ and by $\eqref{sizeofG}$, we have
\begin{equation} \label{G1}
|G| \ = \sum_{\vec{n} \in S}  r(\mathscr{A}^s;\vec{n}) \ \geq \ \frac{\alpha}{2} N^{s-\delta}.
\end{equation}
From the definition of $\Sigma(G)$ in $\eqref{sigdef}$, we note that
\begin{equation} \label{SumG1}
|\Sigma(G)| = |S| \ \leq \ \frac{4}{\alpha} N^{k(k+1)/2}.  
\end{equation}
In the next section, we prove the following lemma about hypergraphs on $\mathscr{A}^s$ which will deliver our result.
\begin{lemma} \label{hypvmvt}
For every $m \geq 1$, $(s,k) \in \mathbb{N}^2$ such that $s \geq k(k+1)$, and sufficiently small positive number $\epsilon$, there exists $\delta >0$, such that the following holds. Suppose that $A$ is a finite well-spaced subset of real numbers such that $|A|$ is sufficiently large and $X_A \leq {|A|}^m$. Let $\alpha$ be any real number such that
\[ |A|^{- k(k+1) \epsilon /43200} \ \leq \ \alpha \ \leq 1 .  \]
If there exists $G \subseteq \mathscr{A}^s$ such that
\[  |G| \geq \frac{\alpha}{2} N^{s-\delta} \  \ \text{and} \  \ |\Sigma(G)| \leq  \frac{4}{\alpha} {|A|}^{k(k+1)/2},  \]
then there exists $A' \subseteq A$ with $|{A'}| \ \geq \ {\alpha}^{\epsilon/4} |A|^{1 - \epsilon k^2}$,
such that for all $l \in \mathbb{N}$, we have
\[ |l \mathscr{A'}| \ \leq \  {\alpha}^{-(1 + 2\epsilon(l+k^2))}  |A'|^{\frac{k(k+1)}{2} (1 + \epsilon( l +  k^2 ))} , \]
where $\mathscr{A'} = \{ (a,a^2,\dots,a^k) \ | \ a \in A'\}$.
\end{lemma}
 Note that the discussion leading to $\eqref{G1}$ and $\eqref{SumG1}$ combines with Lemma $\ref{hypvmvt}$ to prove Theorem $\ref{bsgvmvt}$. Thus, it suffices to show Lemma $\ref{hypvmvt}$, which is what we will do next. Throughout the proof of Lemma $\ref{hypvmvt}$, we will work with a given $\epsilon > 0$, and show that the corresponding result holds for each $\delta > 0$ which is sufficiently small in terms of $\epsilon$.

\section{Proof of Lemma $\ref{hypvmvt}$} \label{prooflm}

All that remains is to establish Lemma $\ref{hypvmvt}$, which is our primary focus in this section. We follow Borenstein and Croot's work from \cite{BC2011} closely in various parts of this section. 
\par

It will be ideal to describe elements in $\mathscr{A}^s$ as strings of length $s$ with alphabets from $\mathscr{A}$. Thus if $x \in \mathscr{A}^s$, then $x$ has a string representation as $x = a_1a_2 \dots a_s$, where $a_i \in \mathscr{A}$ for all $1 \leq i \leq s$. When $s$ is even, given a string $x$ of length $s/2$, we can define its right neighbourhood $R(x)$ as
\begin{equation*} R(x) = \{ y \in \mathscr{A}^{s/2} \ | \ xy \in G \}. \end{equation*}
Lastly, given a string $x \in \mathscr{A}^s$ such that $ x = a_1a_2 \dots a_s$, we can define 
$$ \Sigma(x) = \sum_{i=1}^{s} a_i .$$
\par

We consider $G \subseteq \mathscr{A}^s$ as described in the hypothesis of Lemma $\ref{hypvmvt}$. Let $s \geq k(k+1)$. We first show that the result follows for $s$ odd from the case of $s$ even. If $s$ is an odd number such that $s > k(k+1)$, we can write every string $x$ in $G$ as $x = x'  a_x$ such that $x'$ is a string of length $s-1$ and $a_x \in \mathscr{A}$. For each $a \in \mathscr{A}$, we can define $G_a \subseteq \mathscr{A}^{s-1}$ such that $x'a \in G$ for all $x' \in G_a$. By double counting, there must exist an $a \in \mathscr{A}$ such that 
$$ |G_a| \ \geq \ \frac{|G|}{|\mathscr{A}|} \ \geq \ \frac{\alpha}{2} N^{s-1 - \delta}.$$ 
Fixing such an $a$, we further show that
\begin{align*}
  | \Sigma(G_a)| & = | \{ a_1 + \dots + a_{s-1} \ | \ (a_1, \dots,a_{s-1}) \in G_a  \}| &\\ & \leq | \{ a_1 + \dots + a_{s-1} + a \  |  \ (a_1, \dots,a_{s-1}) \in G_a  \}|  &\\ &  \leq | \{ a_1 + \dots + a_{s-1} + a_{s} \ |  \ (a_1, \dots,a_{s-1},a_{s}) \in G  \}| &\\ & = |\Sigma(G)| \leq  \frac{4}{\alpha} N^{k(k+1)/2}.& \end{align*}
Now we can apply  Lemma $\ref{hypvmvt}$ for $G_a \subseteq \mathscr{A}^{s-1}$, wherein $s-1$ is an even number such that $s-1 \geq k(k+1)$, to get the desired conclusion. Thus, we can assume without loss of generality that $s$ is even.
\par

By double-counting, we get
\begin{equation} \label{dct1}
\sum_{x \in  \mathscr{A}^{s/2}} |R(x)| = |G| \geq \frac{\alpha}{2} N^{s-\delta}.
\end{equation}
\begin{lemma} \label{popl}
Let $V$ be a set of $n$ elements and $U_1,U_2,\dots,U_r \subseteq V$ such that
$$ \sum_{i=1}^{r} |U_i| \geq C r n^{1-\delta},$$
where $C$ is some positive number.
Then there exists $j \in \{1,2,\dots,r\}$ such that
$$ \sum_{i=1}^{r} |U_i \cap U_j| \geq C^2 rn^{1-2\delta}.$$
\end{lemma}
\begin{proof}
Let ${1}_{A}(x)$ be the indicator function of set $A$. Thus, we have
$$ \sum_{x \in V} \sum_{i=1}^{r} {1}_{U_i}(x) = \sum_{i=1}^{r} \sum_{x \in V} {1}_{U_i}(x) =  \sum_{i=1}^{r} |U_i| \geq C r n^{1-\delta}. $$
Applying the Cauchy-Schwarz inequality on the above, we get
$$ \sum_{x \in V} 1 \sum_{x \in V} \bigg( \sum_{i=1}^{r} {1}_{U_i}(x) \bigg)^2 \geq C^2 r^2 n^{2-2\delta},$$
which in turn gives us
$$\sum_{x \in V} \sum_{j=1}^{r} \sum_{i=1}^{r}{1}_{U_j}(x){1}_{U_i}(x)  \geq C^2 r^2 n^{1-2\delta}. $$
Interchanging the summations, we get
$$\sum_{j=1}^{r} \sum_{i=1}^{r} |U_i \cap U_j| \geq  C^2 r^2 n^{1-2\delta},  $$
from which we can deduce the existence of some $j \in \{ 1,2,\dots,r\}$ such that 
$$ \sum_{i=1}^{r} |U_i \cap U_j| \geq C^2  rn^{1-2\delta}.$$
\end{proof}

Applying Lemma $\ref{popl}$ to $\eqref{dct1}$, we show that there exists an $x \in \mathscr{A}^{s/2}$ such that 
\begin{equation}\label{rxcapry} \sum_{y \in \mathscr{A}^{s/2}} |R(x) \cap R(y)| \geq \frac{{\alpha}^2}{4} N^{s-2\delta}. \end{equation}
Fixing such an $x$, let $G_1 = \{ yz \in G \ | \ z \in R(x)\}$. Note that by definition of $G_1$, $$|G_1| = \sum_{y \in \mathscr{A}^{s/2}} |R(x) \cap R(y)| \geq \frac{ {\alpha}^2}{4} N^{s-2\delta}. $$
\begin{lemma} \label{idkw}
Given $\epsilon > 0$, the following holds for all $\delta >0$ sufficiently small with respect to $\epsilon$. Suppose $X \subseteq \mathscr{A}^{s/2}$ such that $$|X| \geq \frac{\alpha^4}{2^5} N^{s/2 - 8\delta}.$$ Then we have
\begin{equation} \label{sizeofry}
|\Sigma (X)| \geq |\Sigma(G_1)| ^{1- \epsilon/1200}.
\end{equation}
\end{lemma}
\begin{proof}
We will prove $\eqref{sizeofry}$ by contradiction. We start by defining a variant of $r(\mathscr{A}^s;\vec{n})$. For $\vec{n} \in \mathbb{R}^k$, we write
$$r(X;\vec{n}) = | \{ (a_1,a_2,\dots,a_{s/2}) \in X \ | \ \vec{n} =  \sum_{i=1}^{s/2} a_i \ \}|. $$
We now use the Cauchy-Schwarz inequality to get
$$  |X|^2 \ = \ \big( \sum_{\vec{n} \in \Sigma(X)} r(X ;\vec{n}) \big)^2  \leq  \sum_{\vec{n} \in \Sigma(X)} r(X ;\vec{n})^2  \sum_{\vec{n} \in \Sigma(X)}1 \ .   $$
From the above expression, we have
$$ J_{s/2,k}(A)  \geq \sum_{\vec{n} \in \Sigma(X)} r(X ;\vec{n})^2  \geq \frac{|X|^2}{|\Sigma(X)|}.$$
Thus if $\eqref{sizeofry}$ does not hold, that is, if $$ |\Sigma (X)| < |\Sigma(G_1)| ^{1- \epsilon/1200} \leq |\Sigma(G)| ^{1- \epsilon/1200} \leq \bigg(\frac{4}{\alpha}\bigg)^{1- \epsilon/1200} N^{\frac{k(k+1)}{2}({1- \epsilon/1200})},  $$
then we can show that
\begin{align*} J_{s/2,k}(A)  
& > \bigg(\frac{\alpha}{4}\bigg)^{1- \epsilon/1200}  {|X|^2}{N^{-\frac{k(k+1)}{2}({1- \epsilon/1200})}} \\ 
& \geq \bigg(\frac{\alpha}{4}\bigg)^{1- \epsilon/1200}  \frac{{\alpha}^8}{2^{10}}  {N^{s- 16\delta}}{N^{-\frac{k(k+1)}{2}({1- \epsilon/1200})}} \\ 
& \geq  \frac{{\alpha}^{9 -\epsilon/1200}}{2^{12}} {N^{s- k(k+1)/2}}{N^{{ \frac{k(k+1)}{2400}\epsilon  - 16\delta}}}.
 \end{align*}
The hypothesis of Lemma $\ref{hypvmvt}$ permits us the assumption that 
$$\alpha \geq N^{-k(k+1) \epsilon /43200} .$$ 
Thus we have
\begin{equation} \label{est}  J_{s/2,k}(A)  \ > \ N^{s-k(k+1)/2} N^{\frac{k(k+1)}{9600} \epsilon}                                    ,\end{equation}
whenever $N$ is sufficiently large and $\delta$ is sufficiently small as compared to $\epsilon$. But observe that as $s/2 \geq k(k+1)/2$, the upper bound $\eqref{eq:uppbd1}$ holds true for $J_{s/2,k}(A)$ when $N$ is large enough in terms of $\epsilon, s$ and $k$, which contradicts $\eqref{est}$. Thus $\eqref{sizeofry}$ must hold.
\end{proof}
We will now work with strings $y \in   \mathscr{A}^{s/2}$ that have a large $R(y) \cap R(x)$. Thus, we let 
 $$ Y = \{ y \in   \mathscr{A}^{s/2} \ | \  |R(y) \cap R(x)| \ \geq \frac{{\alpha}^2}{4}  N^{s/2 - 4\delta}\}.$$
We can infer from Lemma $\ref{idkw}$ that for each $y \in Y$, we have
\begin{equation} \label{idkk} |\Sigma (R(y) \cap R(x))| \geq |\Sigma(G_1)| ^{1- \epsilon/1200}. \end{equation}
We observe that
\[   \sum_{y \notin Y} |R(x) \cap R(y)|   <  \frac{{\alpha}^2}{4}  N^{s/2 - 4\delta} \sum_{y \notin Y}1  \leq  \frac{{\alpha}^2}{4}  N^{s/2 - 4\delta} N^{s/2}.  \]
Using this with $\eqref{rxcapry}$, we get
\[ \sum_{y \in Y}|R(x) \cap R(y)| =  \sum_{y \in \mathscr{A}^{s/2}} |R(x) \cap R(y)| - \sum_{y \notin Y} |R(x) \cap R(y)| \geq \frac{{\alpha}^2}{4} N^{s-2\delta} - \frac{{\alpha}^2}{4}  N^{s - 4\delta} .   \]
Thus when $N$ is sufficiently large, we have 
\[ |Y| N^{s/2} \geq \sum_{y \in Y}|R(x) \cap R(y)| \geq \frac{{\alpha}^2}{4} N^{s-2\delta} - \frac{{\alpha}^2}{4}  N^{s - 4\delta}  \geq  \frac{{\alpha}^2}{4} N^{s- 4\delta},       \]
which implies that
$$ |Y| \geq \frac{{\alpha}^2}{4} N^{s/2- 4\delta}. $$

\par
Note that as
$$ \sum_{z \in R(x)} | \{ y \in Y \ | \ yz \in G_1 \} | \ = \  \sum_{y \in Y} |R(y) \cap R(x)| \ \geq \ \frac{{\alpha}^2}{4} |Y| N^{s/2 - 4\delta} \  \geq  \ \frac{{\alpha}^4}{2^4} N^{s- 8\delta},$$
there must exist some $z \in R(x)$ such that 
\begin{equation} \label{idk4}  | \{ y \in Y \ | \ yz \in G_1 \} |  \ \geq \ \frac{{\alpha}^4}{2^4} N^{s/2 - 8\delta}. \end{equation}
Fixing such a $z$, we let $Y_1 =  \{ y \in Y \ | \ yz \in G_1 \}$. We will prune $Y_1$ further to get the set of all popular representations in $Y_1$. To this end, we will look at the highly-represented sums, that is, let $S_1 \subseteq \Sigma(Y_1)$ such that
$$ S_1 = \bigg\{ \vec{n} \in  \Sigma(Y_1) \ | \   r(Y_1;\vec{n}) >  \frac{|Y_1|}{2 |\Sigma(Y_1)| \  }  \bigg\},$$
and define $Y_2 \subseteq Y_1$ as follows
$$ Y_2 = \{  y \in Y_1 \ | \  \Sigma(y) \in S_1 \}.$$
We can get a lower bound for popular representations by bounding the number of unpopular representations. Thus
$$ |Y_1 \setminus Y_2| = \sum_{\vec{n} \in \Sigma(Y_1\setminus Y_2)}  r((Y_1 \setminus Y_2);\vec{n}) \ \leq \ | \Sigma(Y_1\setminus Y_2)|  \frac{|Y_1|}{2 |\Sigma(Y_1)|} \ \leq  \frac{|Y_1|}{2}.$$
Using the above upper bound along with $\eqref{idk4}$, we can deduce that
$$ |Y_2| \geq  \frac{|Y_1|}{2} > \frac{{\alpha}^4}{2^5} N^{s/2 - 8\delta}. $$
As $Y_2$ is a large subset of $\mathscr{A}^{s/2}$, by Lemma $\ref{idkw}$ , we have
\begin{equation} \label{idk5}  |\Sigma (Y_2)| \geq |\Sigma(G_1)| ^{1- \epsilon/1200}. \end{equation}
Our set up for the proof of Lemma $\ref{hypvmvt}$ is complete. 

\begin{Proposition} \label{prebsg}
There exist sets $Z_1 \subseteq \Sigma(Y_2)$ and $Z_2 \subseteq  \Sigma(R(x))$ such that 
\begin{equation} \label{bwrt}   
|\Sigma(G_1)|^{1- \epsilon /1200} \leq  |Z_1| = |Z_2| \leq  |\Sigma(G_1)|    
\end{equation}
and $$E(Z_1,Z_2) \ \geq \ \frac{1}{4 |\Sigma(G_1)|^{\epsilon / 200}} \ |Z_1|^{3}.$$ 
\end{Proposition}

\begin{proof}
We know that 
\begin{equation} \label{ratio}  |\Sigma(G_1)|^{1- \epsilon /1200} \leq |\Sigma(Y_2)|  \leq  |\Sigma(G_1)|,  \end{equation}
and
\begin{equation} \label{ratio1}  |\Sigma(G_1)|^{1- \epsilon /1200} \leq |\Sigma(R(x))| \leq  |\Sigma(G_1)|.  \end{equation}
For ease of notation, if $|\Sigma(Y_2)| \geq  |\Sigma(R(x))|$, let $U = \Sigma(Y_2)$, $V = \Sigma(R(x))$, otherwise let $U =  \Sigma(R(x))$ and $V = \Sigma(Y_2)$. Further, for $\vec{n} \in \mathbb{R}^k$ and sets $A,B \subseteq \mathbb{R}^k$, define
$$ r(A,B;\vec{n}) = | \{ (a,b) \in A\times B \ | \ \vec{n} = a + b\}|. $$
\par

Split $U$ into a disjoint union of subsets $U_1, U_2, \dots, U_r$ such that $$|V| = |U_1| = |U_2| = \dots = |U_{r-1}| \geq |U_r|.$$
By using the Cauchy-Schwarz inequality, we see that
\begin{align} 
 \bigg(\sum_{\vec{n} \in \Sigma(G_1)} r(U,V;\vec{n}) \bigg)^2 \ 
& = \ \bigg(\sum_{i=1}^{r} \sum_{\vec{n} \in \Sigma(G_1)} r(U_i,V;\vec{n})\bigg)^2   \nonumber \\ 
&  \leq \ r |\Sigma(G_1)|  \sum_{i=1}^{r}\sum_{\vec{n} \in \Sigma(G_1)} r(U_i,V;\vec{n})^2  \nonumber \\  
& =  \ r \ |\Sigma(G_1)| \ \sum_{i=1}^{r} E(U_i,V) \nonumber \\ 
&  \leq \ r^2  |\Sigma(G_1)| \ \max_{1\leq i \leq r}  E(U_i,V).   \label{idkbp} 
\end{align}
Moreover, we have
\begin{equation*}  \begin{split} 
\sum_{\vec{n} \in \Sigma(G_1)} r(U,V;\vec{n}) \ 
& = \ \sum_{\vec{n} \in \Sigma(G_1)} | \{ (a,b) \in \Sigma(Y_2)\times \Sigma(R(x)) \ | \ \vec{n} = a+b  \}| \\  
& \geq \ |\Sigma(Y_2)| \ {\min_{y \in Y_2}} \  | \sum( R(y) \cap R(x))| .
\end{split} \end{equation*}
Note that as $Y_2 \subseteq Y$, the bound $\eqref{idkk}$ holds for all $y \in Y_2 $. Using this along with $\eqref{ratio}$, we deduce that
\[    \sum_{\vec{n} \in \Sigma(G_1)} r(U,V;\vec{n}) \ \geq \ |\Sigma(Y_2)| \ {\min_{y \in Y_2}} \  | \sum( R(y) \cap R(x))| \ \geq \    \big|\Sigma(G_1) \big|^{2- \epsilon/600}. \] 
Combining the above bounds with $\eqref{idkbp}$, we get
$$r^2  |\Sigma(G_1)| \ \max_{1\leq i \leq r}  E(U_i,V) \ \geq |\Sigma(G_1)|^{4- \epsilon /300}.$$
If $r=1$, then we have
\[\max_{1\leq i \leq r}  E(U_i,V)  \ \geq \ |\Sigma(G_1)|^{3 - \epsilon / 200}. \]
If $r \geq 2$, then 
\[|U| = \sum_{i=1}^r |U_i| \geq \sum_{i=1}^{r-1} |U_i| \geq \frac{r}{2} |V|,\]
and noting $\eqref{ratio}$ and $\eqref{ratio1}$, we get that
\[ r \leq 2 \frac{|U|}{|V|} \leq 2|\Sigma(G_1)|^{\epsilon/1200}.\] 
Thus we have
\begin{equation} \label{ref1}
\max_{1\leq i \leq r}  E(U_i,V)  \ \geq \frac{1}{4} |\Sigma(G_1)|^{3 - \epsilon / 200}.
\end{equation}
\par

In either case, $\eqref{ref1}$ holds. Let $j \in \{1,2,\dots,r\}$ such that $E(U_j,V) = \max_{1\leq i \leq r}  E(U_i,V)$. We set $Z_1 = V$. If $j \neq r$, set $Z_2 = U_j$. If $j=r$, set $Z_2 = U_r \cup Z_3$ where $Z_3$ is some subset of $U \setminus U_r$ such that $|Z_3| = |V| - |U_r|$. In both the cases, we get two sets $Z_1, Z_2$ such that $|Z_1| = |Z_2|$. Moreover, using $\eqref{ref1}$, we see that
$$E(Z_1,Z_2) \ \geq \ \frac{1}{4} |\Sigma(G_1)|^{3 - \epsilon / 200} \ \geq \  \frac{1}{4 |\Sigma(G_1)|^{\epsilon / 200}} \ |Z_1|^{3} ,$$
and one of $Z_1,Z_2$ is a subset of $\Sigma(Y_2)$, while the other is a subset of $\Sigma(R(x))$.
\end{proof}

We now apply Theorem $\ref{bsg}$ to the sets $Z_1,Z_2$ that we get from Proposition $\ref{prebsg}$. If $N$ is large enough in terms of $\epsilon$, then $1/|\Sigma(G_1)|^{\epsilon/200}$ is small enough and thus, we get a set $S_2 \subseteq Z_1 \subseteq \Sigma(Y_2)$ such that
$$ |S_2| \geq C^{4} |Z_1| \ , \ |S_2 + S_2| \leq \frac{1}{C^{8}} |Z_1| ,$$
where $$C = \frac{1}{4 |\Sigma(G_1)|^{\epsilon / 200}} \ .$$
If $N$ is sufficiently large, we can show that
\begin{equation} \label{idkpfm}
|S_2| \geq |Z_1|^{1 - \epsilon/5} \ , \ |S_2 + S_2| \leq |S_2|^{1 + \epsilon/2}.
\end{equation}
Observe that $S_2 \subseteq \Sigma(Y_2)$. Thus let $Y_3$ be the set of all strings $y \in Y_1$ such that $\Sigma(y) \in S_2$. 
\par

We can estimate the size of $Y_3$, as for any $y \in Y_1$, if $\Sigma(y) \in \Sigma(Y_2) = S_1$, then $y$ is a popular representation. Hence, we have
\begin{align*}
|Y_3| = \sum_{\vec{n} \in S_2} r(Y_1;\vec{n}) \ \geq \ |S_2| \frac{|Y_1|}{2 |\Sigma(Y_1)|}.
\end{align*}
Noting $\eqref{idkpfm}$ and the fact that $\Sigma(Y_1) \leq \Sigma(G_1)$, we get
\[ |S_2| \frac{|Y_1|}{2 |\Sigma(Y_1)|} \ \geq \  \frac{1}{2 |\Sigma(G_1)| \  } |Z_1|^{1 - \epsilon/5} |Y_1| .\]
Next, we use $\eqref{bwrt}$ and $\eqref{idk4}$ to see that
\[\frac{1}{2 |\Sigma(G_1)| \  } |Z_1|^{1 - \epsilon/5} |Y_1| \geq   \frac{\alpha^4}{2^5} |\Sigma(G_1)|^{-\epsilon/4} N^{s/2 - 8\delta} . \] 
The hypothesis of Lemma $\ref{hypvmvt}$ implies that $|\Sigma(G_1)| \leq |\Sigma(G)| \leq 4\alpha^{-1} N^{k(k+1)/2}$ and $\alpha \geq N^{-k(k+1)\epsilon/43200}$. Thus, we deduce that
\[ \frac{\alpha^4}{2^5} |\Sigma(G_1)|^{-\epsilon/4} N^{s/2 - 8\delta} \geq {\alpha}^{\epsilon/4} N^{s/2 - 8\delta - k(k+1)\epsilon /4}, \]
when $N$ is large enough. Combining these inequalities, we get
\[ |Y_3| \geq {\alpha}^{\epsilon/4} N^{s/2 - 8\delta - k(k+1)\epsilon /4}. \] 
\par

As $Y_3 \subseteq \mathscr{A}^{s/2}$, we can deduce from the above inequality that there exists some string $w \in \mathscr{A}^{s/2-1}$ such that $wa \in Y_3$ for at least ${\alpha}^{\epsilon/4} N^{1 - 8\delta - k(k+1)\epsilon /4}$ elements $a \in \mathscr{A}$. Fix such a string $w \in \mathscr{A}^{s/2-1}$, and define
$$ \mathscr{A'} = \{  a \in \mathscr{A} \ | \ wa \in Y_3 \}. $$
From above, we know that
\begin{equation} \label{f1} 
|\mathscr{A'}| \geq  {\alpha}^{\epsilon/4}|\mathscr{A}|^{1 - 8\delta - k(k+1)\epsilon /4} \geq {\alpha}^{\epsilon/4} |\mathscr{A}|^{1 - \epsilon k^2}, 
\end{equation}
if  $\delta$ is sufficiently small in terms of $\epsilon$. Furthermore, for all $l \in \mathbb{N}$ we can see that
\begin{equation} \label{f2} 
 l \mathscr{A'} + l\Sigma(w) \ \subseteq \ l\Sigma(Y_3) \ \subseteq \ l S_2.  
\end{equation}
We have small doubling for the set $S_2$, but we actually need bounds for the $l$-fold sumset $lS_2$. In order to move from strong upper bounds for sumsets to estimates for many fold sets, we will use the Pl{\"u}nnecke-Ruzsa theorem as stated in \cite[Corollary 6.29]{TV2006}.
\begin{lemma} \label{prineq}
Let $A$ be a finite subset of some additive abelian group $G$. If $|A+A| \leq K|A|$, then for all non-negative integers $m,n$, we have
$$ |mA - nA| \leq K^{m+n}|A|.$$
\end{lemma}
We now use Lemma $\ref{prineq}$ on $S_2$ to show that for all $l \in \mathbb{N}$, we have
\begin{align*} |l S_2| \ & \leq \ |S_2|^{1 + \epsilon l/2} \ \leq \  \ | \Sigma(G)|^{1 + \epsilon l/2}  \\ & \leq \ \  \bigg( \frac{4}{\alpha} \bigg)^{1 + \epsilon l/2}  N^{\frac{k(k+1)}{2}(1 + \epsilon l/2)}   \end{align*}
Noting $\eqref{f2}$, we get that
\begin{equation*}   |l \mathscr{A'}| \ \leq \ |l S_2| \ \leq \  \bigg( \frac{4}{\alpha} \bigg)^{1 + \epsilon l/2} N^{\frac{k(k+1)}{2}(1 + \epsilon l/2)}.\end{equation*}
From $\eqref{f1}$, we see that
\[N = |\mathscr{A}| \leq ({\alpha}^{-\epsilon/4} |\mathscr{A'}|)^{1/(1 - 8\delta - k(k+1)\epsilon /4)}.  \]
Thus, we have
\begin{align*}
|l \mathscr{A'}| \ 
& \leq \ \bigg( \frac{4}{\alpha} \bigg)^{1 + \epsilon l/2} N^{\frac{k(k+1)}{2}(1 + \epsilon l/2)} \\
& \leq \  4^{1 + \epsilon l/2} \bigg(  \frac{1}{\alpha}  \bigg)^{P_1} |\mathscr{A'}|^{P_2}, 
\end{align*}
where
\[ P_1 =  \frac{ \epsilon k(k+1)(1 + \epsilon l/2)}{8(1 -  8\delta - k(k+1)\epsilon /4)} + 1 + \epsilon l/2, \]
and
\[ P_2 = \frac{k(k+1)}{2} \frac{(1 + \epsilon l/2)}{(1 -  8\delta - k(k+1)\epsilon /4)}.  \]  
After some elementary computations, we observe that
\[ P_1 = \frac{ \epsilon k(k+1)(1 + \epsilon l/2)}{8(1 -  8\delta - k(k+1)\epsilon /4)} + 1 + \epsilon l/2 \ \leq \ 1 + \epsilon(l+k^2),\]
and
\[ P_2 = \frac{k(k+1)}{2} \frac{(1 + \epsilon l/2)}{(1 -  8\delta - k(k+1)\epsilon /4)}  \  \leq \   \frac{k(k+1)}{2} (1 + \epsilon( l/2  +  k^2 )),   \]
for all $l \geq k(k+1)/2$, whenever $\epsilon$ is sufficiently small and $\delta$ is sufficiently small in terms of $\epsilon$.
Hence, we get
\[ |l \mathscr{A'}|   \leq \  \bigg(  \frac{1}{\alpha}  \bigg)^{1 + 2\epsilon(l+k^2)}  |\mathscr{A'}|^{\frac{k(k+1)}{2} (1 + \epsilon( l +  k^2 ))} \]
whenever $N$ is large enough, $\epsilon$ is sufficiently small and $\delta$ is sufficiently small in terms of $\epsilon$. Thus, we have proven Lemma $\ref{hypvmvt}$.

\section{Proof of Theorem \ref{vmvtsp}}

In this section, we prove Theorem $\ref{vmvtsp}$. Given a set $A \subseteq \mathbb{R} \setminus \{0\}$ satisfying the hypothesis of Theorem $\ref{vmvtsp}$, we can use the pigeonhole principle to find $A' \subseteq A$ such that $|A'| \geq |A|/2$ and 
\begin{equation} \label{psedit}
 a_1 a_2 > 0 \ \text{and} \ a_1/ a_2 > 0
 \end{equation}
for all $a_1, a_2 \in A'$. To prove Theorem $\ref{vmvtsp}$ for the set $A$, it suffices to show Theorem $\ref{vmvtsp}$ for the set $A'$. Thus, we can assume that $\eqref{psedit}$ holds for all $a_1, a_2 \in A$. 
\par

We begin by defining the multiplicative energy $M(A)$ as
\[ M(A) =  \bigg| \bigg\{ (a_1,a_2,a_3,a_4) \in A^4 \ \bigg| \ \frac{a_1}{a_2}  = \frac{a_3}{a_4}  \ \bigg\}  \bigg|.\]
Moreover, for all $x \in \mathbb{R}$, we write
\[ d(x) = |\{ (a_1,a_2) \in A^2 \ | \ x = a_1/a_2 \}|. \]
Note that 
\begin{equation} \label{meps}
 M(A) = \sum_{x \in A/A} d(x)^2, 
\end{equation}
and 
\begin{equation}  \label{basicp}
1 \leq d(x) \leq |A| \ \text{for all} \ x \in A/A .
\end{equation}
As in the previous sections, we can relate $M(A)$ and $|A/A|$ by applying the Cauchy-Schwarz inequality on $\eqref{meps}$. Thus, we have
\begin{equation} \label{mcs}
|A/A| M(A)  \geq \big(\sum_{x \in A/A} d(x) \big)^2 =  |A|^4. 
\end{equation}
Furthermore, from $\eqref{meps}$, we can write $M(A)$ as
\[ M(A) = \sum_{i=0}^{\ceil{\log|A|}} \sum_{\substack{x \in A/A \\ 2^i \leq x < 2^{i+1} }} d(x)^2,\] 
and by the pigeonhole principle, we see that there exists $I \geq 0$ such that
\[M(A) \leq 2{\log|A|} \sum_{\substack{x \in A/A \\ 2^I \leq d(x) < 2^{I+1} }} d(x)^2.\] 
Let 
\[ S = \{ x \in A/A \ | \ 2^I \leq d(x) < 2^{I+1} \} \ \text{and} \ L =  \sum_{x \in S} d(x)^2. \]
Thus we have
\begin{equation} \label{dp1}
 L =  \sum_{x \in S} d(x)^2 \geq M(A)/ 2\log{|A|}.
 \end{equation}
We use $n$ to denote $|S|$. Combining the above with the definition of $S$, we get that
\begin{equation} \label{dp2}
 L = \sum_{x \in S} d(x)^2 \leq n 2^{2I + 2}.
 \end{equation}
\par

From $\eqref{eq:uppbd1}$, we see that
\[ |s \mathscr{A}| \geq |A|^{ s -\epsilon} \ \text{for all} \ k \leq s \leq k(k+1)/2 .\]
Thus if $|A/A| \geq |A|^{1 + \delta_{s}}$, where $\delta_{s} = 1/(2s-1-2\epsilon)$, we are done. Consequently, we can assume that
\begin{equation} \label{thisis}
|A/A| < |A|^{1 + \delta_{s}}.
\end{equation}
\par

Furthermore, we can assume that $ n \geq \log|A|$, as otherwise, noting $\eqref{dp1}$ and $\eqref{basicp}$, we have
\[M(A) < 2n |A|^2  \log|A|  \leq 2|A|^2 (\log|A|)^2. \]
 The above, when combined with $\eqref{mcs}$, implies that
\[ |A/A| \geq |A|^4/ M(A) > |A|^{2} (2\log|A|)^{-2},  \]
which contradicts $\eqref{thisis}$ for $k \geq 2$. Thus we assume that  $ n \geq \log|A|$. 
\par
Moreover, we can see that when $k \geq 2$, we have $\delta_{s} < 1/2$, and consequently, $|A/A| < |A|^{3/2}$. Thus, combining $\eqref{mcs}$, $\eqref{dp2}$ and the fact that $n \leq |A/A| < |A|^{3/2}$, we get
\[ \frac{|A|^4}{2\log|A|} \leq \frac{M(A)}{2\log|A|}|A/A| \leq  L |A/A| < |A|^{3/2} n 2^{2I + 2} \leq |A|^{3}  2^{2I + 2}.  \] 
 This implies that 
\begin{equation} \label{est1}
 2^{I} \geq \frac{|A|^{1/2}}{8\log|A|} .
\end{equation}

We write 
\[ S =  \{ s_1, \dots, s_n \} \]
where $s_i < s_{i+1}$ for all $1 \leq i \leq n-1$. We now consider the sets 
\[ \l_i = \{ (a_1, a_1^2, \dots, a_1^k , a_2, a_2^2, \dots, a_2^k) \ | \  a_2/a_1 = s_i , \ a_1, a_2 \in A \  \}, \]
for $1 \leq i \leq n$. 
\par

\begin{lemma} \label{ul}
For all $u \in \mathbb{N}$ and $1 \leq i < j \leq n$, we have 
\[ |u l_i + u l_j| \gg |ul_i| |u l_j|. \] 
\end{lemma}
\begin{proof}
Let $x_1, x_2 \in u l_i$ and $x_3, x_4 \in u l_j$. Thus
\[  x_v = (\sum_{r=1}^{u} a_{vr},\dots,\sum_{r=1}^{u} a_{vr}^{k},  s_i \sum_{r=1}^{u} a_{vr}, \dots, s_i^k \sum_{r=1}^{u} a_{vr}^k  ) \ \text{for} \ v \in \{1,2\} ,\] 
and 
\[  x_v = (\sum_{r=1}^{u} a_{vr},\dots,\sum_{r=1}^{u} a_{vr}^{k},  s_j \sum_{r=1}^{u} a_{vr}, \dots, s_j^k \sum_{r=1}^{u} a_{vr}^k  ) \ \text{for} \ v \in \{3,4\}, \]
for some $a_{11}, \dots, a_{4u} \in A$. 
\par
If $x_1 + x_3 = x_2 + x_4,$ then $x_1 - x_2 = x_4 - x_3$, which is equivalent to the following system of equations
\begin{align*} 
&  \sum_{r=1}^{u} (a_{1r}^q - a_{2r}^q) = \sum_{r=1}^{u} (a_{4r}^q - a_{3r}^q) \ \ (1 \leq q \leq k) \\
 s_i^q & \sum_{r=1}^{u} (a_{1r}^q - a_{2r}^q) = s_j^q \sum_{r=1}^{u} (a_{4r}^q - a_{3r}^q) \ \ (1 \leq q \leq k).
 \end{align*}
As $i < j$, we have $s_i  < s_j$, and thus, the above system of equations is equivalent to
\begin{align*} 
&  \sum_{r=1}^{u} (a_{1r}^q - a_{2r}^q) = 0 \ \ (1 \leq q \leq k) \\
&  \sum_{r=1}^{u} (a_{4r}^q - a_{3r}^q) = 0 \ \ (1 \leq q \leq k).
 \end{align*}
This implies that $x_1 = x_2$ and $x_3 = x_4$, and thus all the sums of the form $x + y$ with $x \in ul_i$ and $y \in ul_j$, are distinct. This proves Lemma $\ref{ul}$.
\end{proof}

\begin{lemma} \label{vmvtl}
Let $1 \leq u \leq k(k+1)/2$ and $\epsilon$ be a positive real number. Then 
\[ |ul_i| \gg |l_i|^{u - \epsilon}. \]
\end{lemma}

\begin{proof}
We write
\[ p_{i,k} = \{ (a_1, a_1^2, \dots, a_1^k ) \ | \ a_1 \in A \ \text{such that there exists} \ a_2 \in A \ \text{with} \ a_2/a_1 = s_i \}, \]
and
\[ p_i = \{ a_1 \ | \ a_1 \in A \ \text{such that there exists} \ a_2 \in A \ \text{with} \ a_2/a_1 = s_i \}.  \] 
We observe that $|l_i| = |p_{i,k}| = |p_i|$ and $|ul_i| \geq |u p_{i,k}|$. Moreover $p_i \subseteq A$ and thus $X_{p_i} \leq X_A \leq |A|^m$. We now apply Theorem $\ref{BDG1}$ to the set $p_i$ to get
\[ J_{u,k}(p_i) \leq |A|^{\epsilon} |p_i|^{u}, \]
for all $\epsilon > 0$. We apply the Cauchy-Schwarz inequality to this, and get 
\[ |up_{i,k}| \geq |A|^{-\epsilon} |p_i|^{u}.\] 
Noting $\eqref{est1}$, we see that
\[ |p_i| = d(s_i) \geq 2^{I} \gg |A|^{1/3}. \]
Thus, we have
\[ |u l_i| \geq |u p_{i,k}| \gg |A|^{-\epsilon} |p_i|^{u} \gg  |p_i|^{u - 3\epsilon} = |l_i|^{u - 3\epsilon}. \]
We conclude our proof by appropriately rescaling $\epsilon$. 
\end{proof}

We will now consider the sumsets $ul_{i} + ul_{i+1}$ for $1 \leq i \leq n-1$ and prove the following lemma.

\begin{lemma} \label{pnp}
Let $1 \leq i < j \leq n-1$. Then $ul_{i} + ul_{i+1}$ and $ul_{j} + ul_{j+1}$ are disjoint sets. 
\end{lemma} 

\begin{proof}
We prove this lemma by contradiction. Let $(ul_{i} + ul_{i+1})\cap(ul_{j} + ul_{j+1})$ be non-empty. In particular, if we define
\[ m_{r} = \{ (a_1,a_2) \ | \  a_2/a_1 = s_r , \ a_1, a_2 \in A \  \}, \]
for all $1 \leq r \leq n$, then we have that $(um_{i} + um_{i+1})\cap(um_{j} + um_{j+1})$ is non-empty. But if $(x,y) \in \mathbb{R}^2$ such that $(x,y) \in um_{i} + um_{i+1}$, then we have
\[ s_i x < y < s_{i+1} x. \]
Combining the above with the fact that $s_i < s_j$, we see that $(um_{i} + um_{i+1})\cap(um_{j} + um_{j+1}) = \emptyset$. Thus we are done.
\end{proof}

Firstly, we note that \[ ul_i + ul_j \subseteq (u \mathscr{A} + u \mathscr{A}) \times (u \mathscr{A} + u \mathscr{A}), \]
for all $u \in \mathbb{N}$ and $1 \leq i < j \leq n$. With Lemma $\ref{pnp}$ in hand, we have
\[ |u \mathscr{A} + u \mathscr{A}|^2 \ \geq \ \sum_{i=1}^{n-1} | u l_i + u l_{i+1} |. \]
From Lemma $\ref{ul}$, we know that
\[ \sum_{i=1}^{n-1} | u l_i + u l_{i+1} | \ \gg \ \sum_{i=1}^{n-1} |ul_i||u l_{i+1}|. \]
We use Lemma $\ref{vmvtl}$ to get
\[\sum_{i=1}^{n-1} |ul_i||u l_{i+1}| \ \gg \ \sum_{i=1}^{n-1} (|l_i| |l_{i+1}|)^{u - \epsilon}. \] 
We combine this with the fact that $|l_i| = d(s_i) \in [2^{I}, 2^{I+1})$ and $\eqref{dp2}$ to get
\[  \sum_{i=1}^{n-1} (|l_i| |l_{i+1}|)^{u - \epsilon} \ \gg \ n (2^{2I})^{u - \epsilon} \ = \ \frac{n(2^{2I+2}n)^{u - \epsilon}}{(4n)^{u - \epsilon}}  \ \geq \frac{n}{(4n)^{u - \epsilon}} L^{u - \epsilon}. \] 
From the above discussion, we see that
\[ |u \mathscr{A} + u \mathscr{A}|^2 \gg \frac{n}{(4n)^{u - \epsilon}} L^{u - \epsilon}. \] 
This gives us
\[  |u \mathscr{A} + u \mathscr{A}|^{2/(u - \epsilon)} n^{1 - 1/(u - \epsilon)} \gg L. \]
\par

As in the hypothesis of Theorem $\ref{vmvtsp}$, we choose $2 \leq k \leq u \leq k(k+1)/2$. Combining the above with $\eqref{dp1}$ and the fact that $n \leq |A/A|$, we get
\[  |u \mathscr{A} + u \mathscr{A}|^{2/(u - \epsilon)} |A/A|^{1 - 1/(u - \epsilon)} \log|A|  \gg M(A) \]
Using $\eqref{mcs}$ with this, we get
\[ |A|^4 \leq M(A) |A/A| \ll |u \mathscr{A} + u \mathscr{A}|^{2/(u - \epsilon)} |A/A|^{2 - 1/(u - \epsilon)} \log|A|,  \] 
which gives us
\[ |A|^{2u - 2\epsilon}  \ll  |u \mathscr{A} + u \mathscr{A}| |A/A|^{u - 1/2 - \epsilon}.  \]
Hence, we have proven Theorem $\ref{vmvtsp}$.

\section{Variants and Further Applications}

In this section, we show a variant of Theorem $\ref{bsgvmvt}$. We define $\mathcal{U} \subseteq \mathbb{R}$ to be the set of all $\lambda$ such that
\[ J_{s,k}(A) \leq |A|^{2s- k(k+1)/2} |A|^{\lambda}  \]
for all $(s,k) \in \mathbb{N}^2$ satisfying $s \geq k(k+1)/2$ and for all finite, non-empty sets $A \subseteq \mathbb{Z}$ where $|A|$ is large enough with respect to $\lambda$, $s$ and $k$. We note that $\mathcal{U}$ is non-empty as 
\[ J_{s,k}(A) \leq |A|^{2s-k}\]
for all choices of $s,k$ with $s \geq k(k+1)/2$ and $A \subseteq \mathbb{Z}$. Thus $k(k+1)/2 - k \in \mathcal{U}$. We write
\[ \Lambda = \inf \mathcal{U}. \] 
In other words, $\Lambda$ is the smallest real number such that for any choice of $s,k \in \mathbb{N}^2$ satisfying $s \geq k(k+1)/2$ and $\epsilon > 0$, we have
\begin{equation} \label{lamdef}
 J_{s,k}(A) \leq |A|^{2s-k(k+1)/2} |A|^{\Lambda + \epsilon} 
 \end{equation}
for all sets $A$ such that $|A|$ is large enough with respect to $\Lambda$, $s$, $k$ and $\epsilon$. Furthermore, we see that $\Lambda \geq 0$ from the note following $\eqref{lowbd}$. Given a set $A$ with a large $J_{s,k}(A)$, we can prove a result similar to Theorem $\ref{bsgvmvt}$.

\begin{theorem} \label{absvmvt}
Let $s,k \in \mathbb{N}^2$ such that $s \geq k(k+1)$, let $\epsilon$ be a sufficiently small positive real number and let $\Lambda$ be defined as above. Let $A \subseteq \mathbb{Z}$ be a non-empty, finite set such that $|A|$ is large enough and let $\alpha \leq 1$ be any real number such that 
\[ \log{\alpha^{-1}}  \leq \frac{ \epsilon(k(k+1)- 2\Lambda)}{43200}\log{|A|}  \ \text{and} \ J_{s,k}(A) = \alpha |A|^{2s-k(k+1)/2 + \Lambda}. \]
Then there exists $A' \subseteq A$ with $|{A'}| \geq {\alpha}^{\epsilon/4} |A|^{1 - \epsilon k^2}$,
such that for all $l \in \mathbb{N}$, we have
\[ |l \mathscr{A'}| \ \leq \  {\alpha}^{-(1 + 2\epsilon(l+k^2))}  |A'|^{(\frac{k(k+1)}{2} - \Lambda) (1 + \epsilon( l +  k^2 ))} , \]
where $\mathscr{A'} = \{ (a,a^2,\dots,a^k) \ | \ a \in A'\}$. 
\end{theorem}

The proof of Theorem $\ref{absvmvt}$ follows along the same lines as the proof of Theorem $\ref{bsgvmvt}$ with a few minor changes. This mainly consists of using estimate $\eqref{lamdef}$ instead of $\eqref{eq:uppbd1}$ when trying to prove $\eqref{rep}$ and Lemma $\ref{idkw}$, which changes the $k(k+1)/2$ term to $k(k+1)/2 - \Lambda$ in all the suitable places. This is reflected in the conclusion of Theorem  $\ref{absvmvt}$. Moreover, as the condition $X_A \leq |A|^{m}$ is required for the validity of $\eqref{eq:uppbd1}$ but not for $\eqref{lamdef}$, we do not need to impose diameter constraints on our choice of sets in Theorem $\ref{absvmvt}$. \par

Theorem $\ref{absvmvt}$ has applications similar to Theorem $\ref{main}$ as well. We first define the $l$-fold product set $A^{(l)}$ of $A$ to be 
\[A^{(l)}  = \{  a_1 a_2 \dots a_l \ | \ a_i \in A \ (1 \leq i \leq l) \}. \] 
Let $\epsilon> 0$ and $l \in \mathbb{N}$ be parameters that will be fixed later. Let $A'$ be a finite subset of integers such that $|A'|$ is large enough and 
\begin{equation} \label{assum1}
 |l \mathscr{A'}| \ \leq \    |A'|^{(k(k+1)/2 - \Lambda)(1 + \epsilon( l +  k^2 ))}.
 \end{equation}
In this case, we make a simple observation that
\[ |l A'| \leq |l \mathscr{A'}| \leq |A'|^{(k(k+1)/2 - \Lambda)(1 + \epsilon( l +  k^2 ))}. \] 
We will now use the following result of Bourgain and Chang \cite{BC2005}.

\begin{lemma} \label{bcsp}
For all $b \in \mathbb{N}$, there exists $l = l(b) \in \mathbb{N}$ such that if $A \subseteq \mathbb{Z}$ is any finite
set, with $|A| \geq 2$, then 
\[ \sup{(|A^{(l)}|, |lA|)} > |A|^{b}. \]
Moreover, $l$ can be chosen to be $C^{b^4}$ for some constant $C$. 
\end{lemma}

We set $b = \ceil{ k(k+1) - 2\Lambda}$ and $l = C^{b^4}$ in $\eqref{assum1}$, and thus, see that 
\[ |l A'| \leq |A'|^{(k(k+1)/2 - \Lambda)(1 + \epsilon( l +  k^2 ))} < |A'|^{b},  \]
when $\epsilon$ is small enough with respect to $k$ and $\Lambda$.  
Using Lemma $\ref{bcsp}$, we infer that
\[ |{A'}^{(l)}| > |A'|^{b}. \] 
We now apply Lemma $\ref{prineq}$ to get
\[ (|A'A'||A'|^{-1})^{l} \geq  |{A'}^{(l)}||A'|^{-1} > |A'|^{b-1}, \]
which further implies that
\[ |A'A'|> |A'|^{1 + (b-1)/l} .\] 
We note that by definition of $\Lambda$, we have $\Lambda \leq k(k+1)/2 - k$, which implies that $b \geq 2k$, and thus $(b-1)/l$ is a positive real number depending only on $k$ and $\Lambda$. Suppose $A' \subseteq A$ such that $|A'| \geq |A|^{1 - \epsilon k^2}$. Then we would have
\[ |AA| \geq |A'A'| > |A'|^{1 + (b-1)/l} \geq |A|^{(1 + (b-1)/l)(1- \epsilon k^2)} \geq |A|^{1 + (b-1)/2l},\] 
when $\epsilon$ is small enough. We combine this with the conclusion of Theorem $\ref{absvmvt}$ and record this below.

\begin{theorem} \label{absmain}
Let $s,k \in \mathbb{N}^2$ such that $s \geq k(k+1)$, let $\epsilon$ be a sufficiently small positive real number and let $\Lambda$ be defined as earlier. Let $A \subseteq \mathbb{Z}$ be a non-empty, finite set such that $|A|$ is large enough and let $\alpha \leq 1$ be any real number such that 
\[ \log{\alpha^{-1}}  \leq \frac{ \epsilon(k(k+1)- 2\Lambda)}{43200}\log{|A|}  \ \text{and} \ J_{s,k}(A) = \alpha |A|^{2s-k(k+1)/2 + \Lambda}. \]
Then we have
\[ |AA| > |A|^{1 + \delta}, \]
where $\delta$ is some positive constant only depending on $k$ and $\Lambda$. 
\end{theorem}

Theorem $\ref{absmain}$ states that if a set $A$ has extremally large $J_{s,k}(A)$, then it can not have a small product set $AA$. We remark that if we change our definition of set $\mathcal{U} \subseteq \mathbb{R}$ to be the set of all $\lambda$ such that
\[ J_{s,k}(A) \leq |A|^{2s - k(k+1)/2} |A|^{\lambda}\] 
for all finite, large enough sets $A \subseteq \mathbb{Z}$ satisfying $X_A \leq |A|^{m}$ for some fixed $m$, then Theorem $\ref{BDG1}$ implies that $\Lambda = \inf \mathcal{U} = 0$. Thus Theorem $\ref{absvmvt}$ can be seen as a generalisation of Theorem $\ref{bsgvmvt}$ for arbitrary subsets of integers.

%\begin{lemma} \label{rbds}
%Let $A \subseteq \mathbb{Z}$ be a finite, non-empty set. Then for all $\delta > 0$ and $s \geq k(k+1)$, we have
%\[ r(\mathscr{A}^s) \leq |A|^{s-k(k+1)} J_{k(k+1)/2,k}(A).\] 
%Combining Lemma $\ref{rbds}$ with the $\eqref{lamdef}$, we get 
%\[ r(\mathscr{A}^s) \leq |A|^{s-k(k+1)}|A|^{\Lambda + \epsilon} \]
% for all $\epsilon > 0$ and $A \subseteq \mathbb{Z}$ such that $|A|$ is large enough. 
%\end{lemma}

\bibliographystyle{amsbracket}
\providecommand{\bysame}{\leavevmode\hbox to3em{\hrulefill}\thinspace}

\end{document}